\documentclass{amsart}
\usepackage{graphicx}
\usepackage{amsfonts}
\usepackage{amsmath}
\usepackage{amssymb}
\usepackage{amsthm,color}
\usepackage{epsfig,latexsym}

\newtheorem{Thm}{Theorem}

\newtheorem{Lem}{Lemma}
\newtheorem{Prop}{Proposition}

\newtheorem{Def}{Definition}

\newtheorem{Exa}{Example}

\newcommand{\C}{\mathbb{C}}
\newcommand{\Z}{\mathbb{Z}}
\renewcommand{\O}{{\mathcal O}}
\def\mapright#1{\smash{ \mathop{\longrightarrow}\limits^{#1}}}

\begin{document}

\author{Michelle Snider}
%\title{Positivity of K-theoretic Littlewood-Richardson Numbers: Unifying Two Approaches}
\title[Multiplicity-Free Richardson Subvarieties]{A Combinatorial Approach to Multiplicity-Free Richardson Subvarieties of the Grassmannian}

%\thanks{Supported by an NSF grant.}
\email{msnider@math.ucsd.edu}
\date{\today}

\maketitle

\begin{abstract}
We consider Buch's rule for K-theory of the Grassmannian, in the Schur multiplicity-free cases classified by Stembridge. Using a result of Knutson, one sees that Buch's coefficients are related to M\"obius inversion.
We give a direct combinatorial proof of this by considering the product expansion for Grassmannian Grothendieck
polynomials. We end with an extension to the multiplicity-free cases of Thomas and Yong.
\end{abstract}

\tableofcontents

%---------------------------------------------------------------
\section{Motivation}
 \subsection{Schubert and Richardson varieties}
%---------------------------------------------------------------

We consider the \textbf{Grassmannian}  $Gr_k \C^n:=\{V \leq \C^n \; | \; dim(V)=k\}$. For a partition $\lambda$ contained in a $k \times (n-k)$ box,
consider the path from the northeast corner to the southwest corner of the box that traces the partition. For the standard flag ($C_i = (\ast_1, \ldots, \ast_i, 0, \ldots 0)$), we define the \textbf{Schubert variety} as $$X_\lambda = \{ V \in Gr_k \C^n \;| \; dim (V \cap C_i) \geq \#(\text{ south steps in the first}  \; i \text{ steps of the path })\}.$$

%we define the \textbf{binary string of $\lambda$, $bin(\lambda)$}, from the path from the northeast corner to the southwest corner of the box, where 0 denotes a move west and 1 denotes a move south. For the standard flag ($C_i = (\ast_1, \ldots, \ast_i, 0, \ldots 0)$), we define the \textbf{Schubert variety} as $$X_\lambda = \{ V \in Gr_k \C^n \;| \; dim (V \cap C_i) \geq \#(1's \text{ in first }  \; i \text{ places of }bin(\lambda))\}.$$

 We denote the \textbf{Schubert class} in cohomology as $S_\lambda := [X_\lambda]_H \in H^\star(Gr_k \C^n)$. The set $$\{S_\lambda \; | \; \lambda \subset k\times (n-k) \text{ box} \}$$ forms a basis for $H^\star(Gr_k \C^n)$, where $$S_\lambda S_\mu = \sum c_{\lambda \mu}^{\nu} S_\nu$$ for $|\nu|=|\lambda|+|\mu|$, and $c_{\lambda \mu}^{\nu}$ the Littlewood-Richardson coefficients. This follows from the surjective homomorphism $$\{ \text{ring of symmetric polynomials} \} \twoheadrightarrow  \{H^\star(Gr_k \C ^n)\}$$

\begin{center} $s_{\lambda} \longmapsto $
$\left\{
  \begin{array}{ll}
    S_{\lambda}, & \hbox{if $\lambda$ fits in $k\times (n-k)$ box;} \\
    0, & \hbox{otherwise.}
  \end{array}
\right.$
\end{center}
for Schur functions $s_\lambda$.

 The \textbf{M\"obius function $\mu(\nu)$} is defined recursively on a poset ${\mathcal P}$ as the unique function satisfying
 \begin{displaymath} \sum_{\alpha \geq_{\mathcal P} \nu} \mu_{\mathcal P}(\alpha)=1 \end{displaymath}
The connection of this definition to K-classes is shown in \cite{K}. Since we are primarily interested in working in K-theory, we will use $[A]$ to denote the K-class of a subscheme of $A$, and $[A]_H$ to denote its homology class.

Any subvariety $X$ of a flag manifold is rationally equivalent to a linear combination of Schubert cycles with uniquely determined non-negative integer coefficients \cite{Brion}. We say $X$ is \textbf{multiplicity-free} if these coefficients are 0 or 1.

\begin{Thm}\cite{K} \label{thm:mobiusbrion}
  Let $X$ be a multiplicity-free irreducible subvariety of $G/P$, in the sense of \cite{Brion}, with $[X]_H = \sum_{d\in D} [X_d]_H$. Let $\mathcal P \subseteq W/W_P$ be the set of Schubert varieties contained in $\cup_{d\in D} X_d$ (an order ideal in the Bruhat order on $W/W_P$). Then as an element of $K(G/P)$, $$ [X] = \sum_{X_e \subseteq \bigcup_{d\in D} X_d} \mu_{\mathcal P}(X_e)\ [X_e]. $$
\end{Thm}

We will give an independent combinatorial proof of this fact in the case that $X$ is a multiplicity-free \textbf{Richardson variety} in a Grassmannian, the intersection of a Schubert variety $X_\lambda$ with an opposite Schubert variety $w_0 \cdot X_\mu$, for $w_0$ the longest word. For any $X_\lambda \subset Gr_k \C^n $, let $G_\lambda:=[X_\lambda]$. We have that $\{G_\lambda \; | \; \lambda \subset k \times (n-k) \text{ box}\}$ form a basis for $K(Gr_k \C^n)$. For certain symmetric polynomials $g_\lambda$ which we will define in the next section, we have a surjective homomorphism \cite{Buch}:
$$\{ \text{ring of symmetric functions} \} \twoheadrightarrow \{ K(Gr_k \C^n) \}$$
\begin{center} $g_{\lambda} \longmapsto $ $\left\{
  \begin{array}{ll}
    G_{\lambda}, & \hbox{if $\lambda$ fits in $(k \times n-k)$ box;} \\
    0, & \hbox{otherwise.}
  \end{array}
\right.$
\end{center}
Our main theorem will show that, for a poset ${\mathcal P}$ that we will define,  $$G_\lambda \cdot G_\mu = \sum_{\nu} \mu_{\mathcal P}(G_\nu)\ G_\nu $$ where the sum is over $\nu$ such that $\nu \subseteq (k \times n-k)$ box and $|\nu| \geq |\lambda|+|\mu|$. Our proof will proceed with sign-reversing involutions on this poset, and many reductions in the sizes of the partitions in the product.

%---------------------------------------------------------------
\subsection{Grothendieck Polynomials}
%---------------------------------------------------------------
For finite non-empty sets in $\Z^+$, $a$ and $b$, we say $a<b$ if $\text{max}(a) < \text{min}(b)$, and $a \leq b$ if $\text{max}(a) \leq \text{min}(b)$. For a partition $\lambda$, Buch defined a \textbf{set-valued tableau (SVT)} as a filling of a Young diagram with nonempty sets in $\Z^+$ \cite{Buch}. If each box has a single entry, it is a \textbf{Young tableau}. A tableau is a \textbf{semistandard tableau (SS)} if it is weakly increasing across rows and strictly increasing down columns. The \textbf{superstandard} filling of a tableau is the one in which each box $(i,j)$ has a single entry, $i$ (its row). In all of our examples, we will use numbers smaller than 10, so we can avoid the use of set notation: we use 45 to denote the set $\{4,5\}$.\\

Recall the combinatorial definition for the Schur polynomials, $$ s_\lambda=\sum_{T \in SSYT(\lambda)} x^T .$$ We consider the \textbf{Grothendieck polynomials} of Lascoux and Sch\"utzenberger \cite{LS}. For $\lambda$ a partition, Buch \cite{Buch} gives the formula $$g_\lambda=\sum_{T \in SS-SVT(\lambda)} (-1)^{|T|-|\lambda|} x^T$$ where $$|T|=\sum_{i,j} |T(i,j)|.$$ He proves that this is a special case of the Lascoux-Sch\"utzenberger formula (which we will not need) for $g_\pi$ in the case when $\pi$ is a Grassmannian permutation.

\begin{center} \begin{figure}[htbp]
\epsfig{file=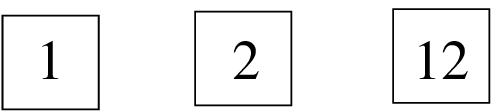, scale=0.4} \caption{In two variables, $g_{\Box}=x_1+x_2-x_1x_2$.}\label{gbox}
\end{figure} \end{center}

These are the $\{ g_\lambda\}$ representing the $G_\lambda$ in the last section. As with the Schur polynomials, it is not obvious from the combinatorial definition that these polynomials are in fact symmetric and a basis for the symmetric polynomials \cite{Buch}. Linear independence follows from the fact that the lowest homogeneous component of $g_\lambda$ is $s_\lambda$.\\

We define the \textbf{word of a tableau $w(T)$} to be the entries read right to left, top to bottom. Note that entries in a set are listed in increasing order, so that they occur in decreasing order in the word. A word is called a \textbf{reverse lattice word (RLW)} if for any initial string, $$multiplicity(i) \geq multiplicity(i+1) \; \forall \; i \geq 1$$ A word that satisfies this condition is sometimes called an \textit{election word}. For tableaux of shape $\lambda$ and $\mu$, we define the shape $\lambda  \times  \mu$ as the skew tableau formed by placing $\mu$ directly southwest of $\lambda$. When we refer to a filling of the shape $\lambda  \times  \mu$, we will call $\lambda$ the ``northeast'' partition, and $\mu$ the ``southwest'' partition.

 Buch \cite{Buch} gives a combinatorial rule for the product of two Grothendieck polynomials: $$g_\lambda g_\mu = \sum {c'}_{\lambda \mu}^{\nu} g_\nu$$
where the coefficients are given by $${c'}_{\lambda \mu}^{\nu} = (-1)^{|\nu|-|\lambda|-|\mu|}\#(T)$$ for SS-SVT $T$ of shape $\lambda  \times  \mu$, content $\nu$, with $w(T)$ a RLW. We call these the K-theoretic Littlewood-Richardson numbers, since if $|\nu| = |\lambda|+|\mu|$, then
${c'}_{\lambda \mu}^{\nu} = c_{\lambda \mu}^{\nu}$, the usual Littlewood-Richardson number.

 \begin{center} \begin{figure}[htbp]
\epsfig{file=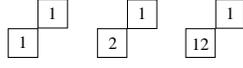, scale=0.4} \caption{$g_1  g_1=g_2+g_{11}-g_{21}$}\label{G1G1}
\end{figure} \end{center}

First, we note that the reverse lattice word condition requires that the filling of the northeast tableau $\lambda$ always be superstandard. We will construct a poset out of all of the possible fillings of the southwest tableau $\mu$, where each vertex is labeled with all tableaux of a given content, and for vertices $\nu,\nu'$, $\nu \leq_{\mathcal P} \nu'$ if $content(\nu) \supset content(\nu')$. Note that for each tableau, the row in the poset corresponds to the number of ``extra" elements in the filling (e.g. the top row has only semistandard Young tableaux). For example, consider the product $G_{2,1} \cdot G_{2,2}$ and its poset in Figure \ref{G21G22}. Note that the product is H-multiplicity-free, but not K-multiplicity-free. The latter cases are extremely rare, occurring only when both partitions $\lambda$ and $\mu$ are rectangles or one of them is a single box or empty (\cite[Proposition 7.2]{Buch}).

\begin{figure}[htbp]
\epsfig{file=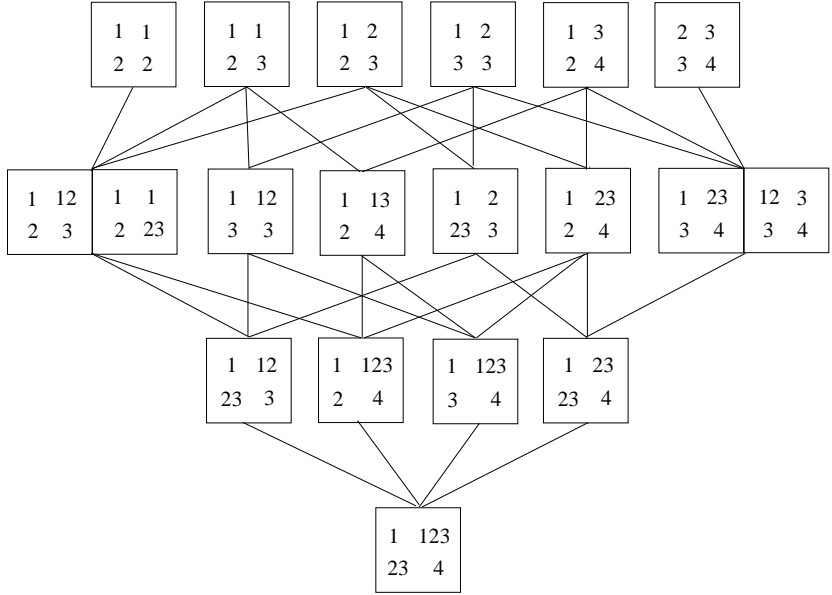, scale=0.75} \caption{The poset corresponding to $G_{2,1} \cdot G_{2,2}$: the product satisfies Stembridge cases (3) and (4) from Theorem \ref{stembridge} below.}\label{G21G22}
\end{figure}

We will consider our products as being inside an ambient box of size $k \times (n-k)$. That is, we limit the terms in the expansion to those indexed by partitions that fit inside this box. We note that this restriction gives us a sub-poset of the full poset. The M\"obius function on the remaining terms is unaffected by the removal of vertices with content exceeding the box size, since all terms above a vertex $\nu$ have content contained in the content of $\nu$. That is, for a given vertex $\nu$ with content in the ambient box, no vertex in its upwards order ideal will have content exceeding the ambient box.
We are interested in cases in which the terms in the Grothendieck expansion which correspond to the Schur expansion are multiplicity-free, i.e. that their coefficients are 0 or 1.

Then Theorem \ref{thm:mobiusbrion} implies the following:
\begin{Thm}\label{mainthm}  Consider partitions $\lambda=(\lambda_1^{\beta_1},\ldots,\lambda_l^{\beta_l})$ and $\mu=(\mu_1^{\alpha_1},\ldots,\mu_m^{\alpha_m})$ such that $G_{\lambda} \cdot G_{\mu}$ in a $k \times (n-k)$ box is a Schur-multiplicity-free product. In the corresponding poset, for each vertex $\nu'$, $\mu(\nu')$ gives the
coefficient of $G_{\nu}$ in the Buch expansion of the product, where $\nu=\nu' \bigcup (1^{\lambda_1},\ldots,l^{\lambda_l})$. \end{Thm}

These Schur-multiplicity-free cases have been classified by Stembridge \cite{Stem}, and our proof explicitly uses his analysis.

We now mention some speculative geometry that motivated our {\em combinatorial} proof of Theorem \ref{mainthm}. Buch shows that the expansion of $X_\lambda \cap (w_0 \cdot X_\mu)$ into Schubert classes has signs that alternate with dimension (\cite{Buch}). This suggests that there exists an exact sequence on sheaves
\begin{equation}\label{eq:pointseq} 0 \rightarrow \O_{\bigcup_{\nu \in {\mathcal P}} X_\nu} \rightarrow \bigoplus_{T,|T|=|\nu|+1} \O_{X_{content(T)}} \rightarrow \cdots \rightarrow  \bigoplus_{T, |T|=|\nu|+ k-1} \O_{X_{content(T)}} \rightarrow \cdots \end{equation}
where the $k^{th}$ nonzero term sums over Buch Littlewood-Richardson tableaux with $k-1$ extra entries. This leads to a sequence for the point in the Grassmannian corresponding to each $\lambda$,
$$0 \rightarrow \C^1 \rightarrow \cdots \rightarrow  \bigoplus_{T, |T|=|\nu|+ k-1, content(T) \subseteq \lambda} \C^1 \rightarrow \cdots $$
One can hope that this sequence is in fact exact.

Our main result is
\begin{Thm}\label{sequencethm} There exists such an exact sequence of vector spaces, and it can be explicitly constructed as a direct sum of exact sequences with exactly two non-zero terms. \end{Thm}
The proof requires an involution which pairs terms differing in size by one. In some cases, we provide a single rule that matches all terms required. In other cases however, we must resort to a multistage divide and conquer approach, where the involution is defined differently on several disjoint subsets. We will come back to this theorem in Section \ref{last}. Assuming Theorem \ref{sequencethm}, we can prove Theorem \ref{mainthm} as a corollary.

\begin{proof}[Our proof of Theorem \ref{mainthm}.]
The exactness of the sequence (\ref{eq:pointseq}) gives us that the alternating sum of dimensions is 0. Thus the sum of the coefficients of the pairs of Buch tableaux, with signs alternating with number of extra numbers, is also 0. Together with the extra 1 from the single fixed point tableau, this is equivalent to the statement that the coefficient of $\nu'$ is given by the M\"obius function.
\end{proof}

%---------------------------------------------------------------
\section{The 2-rectangles Case and the Reduction Lemma}
%---------------------------------------------------------------
We begin by recalling Stembridge's definitions and classification of Schur-multiplicity-free cases.

\begin{Def} \cite{Stem} A partition $\mu$ with at most one part
size (i.e., empty, or of the form $(c^r)$ for suitable $c,r>0$) is said
to be a \textbf{rectangle}. If it has $k$ rows
or $k$ columns (i.e., $k=r$ or $k=c$), then we say $\nu$ is a \textbf{k-line rectangle}.
A partition $\mu$ with exactly two part sizes (i.e., $\mu=(b^rc^s)$
for suitable $b>c>0$ and $r,s>0$) is said to be a \textbf{fat hook}.
If it is possible to obtain a rectangle by deleting a single row or
column from the fat hook $\mu$, then we say that $\mu$ is a
\textbf{near rectangle}.

\end{Def}

We will call these \textbf{top, bottom, left,} or \textbf{right} near rectangles, to denote the location of the extra row or column. We say that a product of Schur functions is \textbf{multiplicity-free} if all of the Littlewood-Richardson coefficients of the expansion are 0 or 1.

\begin{Thm}\label{stembridge} \cite{Stem} The product of Schur functions $s_{\lambda}\cdot s_{\mu}$  is
multiplicity-free if and only if
\begin{enumerate}
\item $\lambda$ and $\mu$ are rectangles, or
\item (Pieri rule) $\lambda$ is arbitrary, and $\mu$ is a
\begin{enumerate}  \item one-row rectangle, or
\item one-column rectangle, or
\end{enumerate}
\item $\lambda$ is a rectangle and $\mu$ is a
\begin{enumerate}
\item left near-rectangle, or
\item bottom near-rectangle, or
\item top near-rectangle, or
\item right near-rectangle, or
\end{enumerate}
\item $\mu$ is a fat hook and $\lambda$ is
\begin{enumerate}  \item two-row rectangle, or
\item two-column rectangle,
\end{enumerate}
\end{enumerate}
\end{Thm}

\noindent (or vice-versa). We note that Stembridge classifies these as 4 cases, but in our analysis they naturally split further and thus have listed them as such. We will start with case (1) and use it to show all remaining cases except for (3d), which requires its own proof technique. We also note that Stembridge's products are not restricted inside a box. Putting our products inside an ambient box gives us a larger class of cases that may not be inherently multiplicity-free, but which lose the terms with multiplicity when considered as being inside a small enough box. These have been classified by Thomas and Yong (\cite{TY}), and we discuss how to generalize to these in the last section of this paper.

Henceforth, $\mu$ will be used to denote the southwest term, and $\lambda$ the northeast partition. In this way, it will be clear of which term we are considering the fillings.

First, let us consider a graphical interpretation of the reverse lattice word condition. We create a diagram by putting $\lambda$ in the northwest corner of the ambient box. As we read the word of $\mu$, for each $i$, we place a box in our diagram in row $i$ adjacent to the rightmost box. Note that the reverse lattice word condition is equivalent to this diagram always being a partition. For values $a$ and $a+1$ in the word of a tableau, we say that  $a+1$ \textbf{depends on} $a$ if the corresponding $a+1$ box occurs directly below the corresponding $a$ box. That is, if that particular $a$ weren't there, we would not be allowed to add that particular $a+1$. For example, in Figure \ref{rectcols}, we say that 3 depends on 2, but 5 does not depend on anything.

\begin{figure}[htbp]
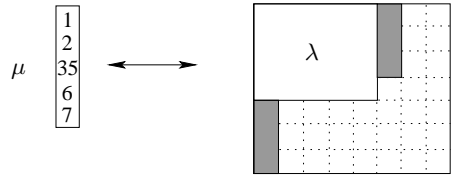 \caption{For $\lambda=(5^3)$ and ambient box $(6 \times 8)$, the correspondence of a column in $\mu=(1^6)$ to a partition.}\label{rectcols}
\end{figure}

\begin{Lem} Consider the product $G_\lambda \cdot G_\mu$ where $\lambda=(\lambda_1^{\beta})$ and $\mu=(\mu_1^\alpha)$ in an ambient box of size $k \times (n-k)$. A filling of a column in tableau $T$ of shape $\mu$ is built out of a basic set of blocks, given by
$$[1,2,\ldots,p]^{tr} \text{ for } p \leq \alpha \text{ and } [\beta+1,\beta+2,\ldots,q]^{tr} \text{ for } q \leq \beta + \alpha$$ where one overlap between the blocks is allowed, and $tr$ indicates transpose. \end{Lem}

\begin{proof} We claim that in terms of the requirements of the reverse lattice word, each column is independent of the others. That is, for each entry, any entries on which it depends occur in the same column. If not, we show that there would be a gap in the column that is not of the form
above. Say we have a column of the form $C=[1,2,..i,i+2,...]^{tr}$. Then in order to satisfy the reverse lattice word condition, there must be a ($i+1)$ in $T$ that comes before $(i+2)$ in $w(T)$, which implies that the $(i+1)$ would be either to the right in the same row (which would violate the semi-standardness condition) or in one of the rows above. Then $(i+1)$ must also come after an $i$ in the word, and either
\begin{itemize}
\item $(i+1)$ depends on the $i$ in $C$, or
\item $(i+1)$ depends on an $i$ in a different column.
\end{itemize}
In the first case, $(i+1)$ would have to be to the left of $i$ in the same row (also violates semi-standardness) or below (same row as $(i+2)$). In the second case $(i+1)$ must have a corresponding  $1,\ldots,i$ preceding it in $w(T)$, but this creates the same problem. A similar argument holds for a gap in $C=[\beta+1,\beta+2,\ldots]^{tr}$. Finally, we note that more than one overlap would violate semi-standardness. \end{proof}

This lemma allows us to take the complicated reverse lattice word condition, and turn it into a simple description of the valid column fillings.

\begin{Exa} Let $\mu=(2,2)$, $\lambda=(3,3,3)$. Then the only possible column fillings for $\mu$ are:

$$\left[{ \begin{array}{c} 1\\          2      \end{array}}\right],
\left[\begin{array}{c}        4\\        5      \end{array}\right],
 \left[\begin{array}{c}        1\\        4      \end{array}\right],
  \left[\begin{array}{c}        14\\        5      \end{array}\right],
   \left[\begin{array}{c}        1\\        24      \end{array}\right].$$

\noindent In Figure \ref{G22G333}, we can see how these columns form tableaux.
\end{Exa}

\begin{figure}[htbp]
\epsfig{file=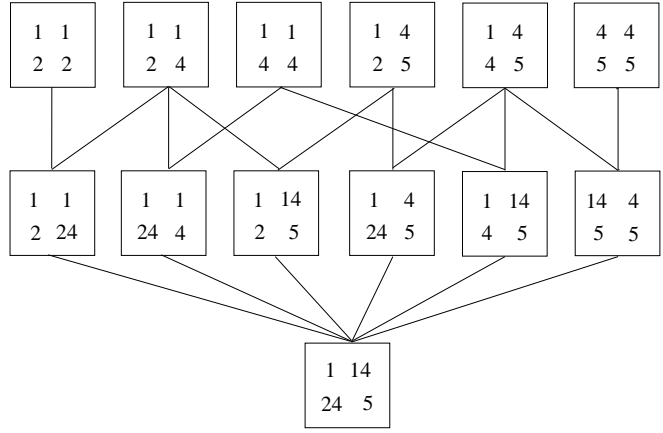, scale=0.75} \caption{The poset corresponding to $G_{2,2} \cdot G_{3,3,3}$. This product satisfies Stembridge case (1).}\label{G22G333}
\end{figure}

 We define the \textbf{snake} of a filling of $\mu$ in the 2-rectangle case as the union of the lines below all single entries $i$ occurring in box $(i,j)$, and the lines above all entries $\lambda_1+1$. The lines are co-linear except for where they enclose entry $i$ in boxes with entry $\{i,\lambda_1+1\}$. Note that a snake uniquely characterizes a filling of a tableau, as the values both above and below the snake are fixed. \\

For a given poset, let $M$ denote the first tableau in lexicographic order, and $M^+$ the upper block defined by the snake. In $M^+$, each box $(i,j)$ has the single entry $i$.  Note that $M$ is a semi-standard Young tableau. In order to prove Theorem \ref{sequencethm}, we will define sign-reversing involutions on the poset to show that we can match all terms except one, $M$. More specifically, our involutions will match terms differing in content size by one and thus terms in adjacent rows. This will show that the K-theoretic Littlewood-Richardson coefficients are given by the M\"obius function.

\begin{Def} We define the function $I_1: \text{SS-SVT} \rightarrow \text{SS-SVT}$ as follows: Compare each box in $M^+$ to the corresponding box in $T$, top to bottom down a column, and left to right across the tableau. If they all match, let $I_(T)=T$. Otherwise, call the first box that doesn't match $(i,j)$.
\begin{itemize} \item If $M(i,j)$ is not in $T(i,j)$, let $T'(i,j) = T(i,j) \cup \{ M(i,j) \}$.
\item  If $M(i,j)$ is in $T(i,j)$, let $T'(i,j)=T(i,j) \backslash M(i,j)$.
\end{itemize}
Then $I_1$ matches $T$ with $T'$. Graphically, applying $I_1$ is equivalent to narrowing or widening the snake by the one box $(i,j)$.\end{Def}

\begin{figure}[htbp]
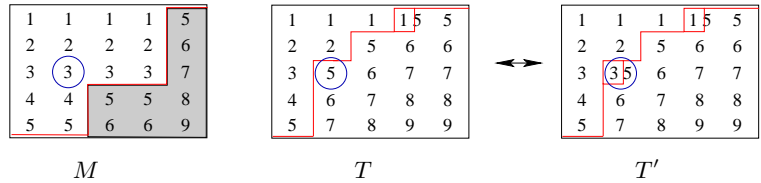 \caption{An example of the involution on terms in the product $G_{4^4} \cdot G_{5^5}$, each term drawn with its snake.}\label{invrect}
\end{figure}

\begin{Prop}\label{2rect} $I_1$ is a sign-reversing involution on the poset corresponding to the product $G_\lambda \cdot G_\mu$ in the case where both partitions are rectangles, $\lambda=(\lambda_1^{\beta})$ and $\mu=(\mu_1^a)$, whose only fixed point is $M$. \end{Prop}

\begin{proof} The choice of box $(i,j)$ is well-defined. In a tableau $T$, if $T(i,j)$ doesn't match $M(i,j)=i$, then either $T(i,j)=[\beta+1]$ or $T(i,j)=[i,\beta+1]$. Since $I_1$ is only dependent on those terms above it in the $j^{th}$ column, adding or removing it will not affect the rest of the filling, and will give another valid tableau.  Since $I_1$ matches a tableau $T$ with another tableau $T'$ that has either one more or one fewer element, it is clearly sign-reversing.
\end{proof}

\begin{figure}[htbp]
\epsfig{file=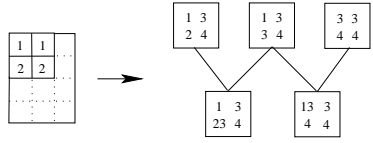, scale=0.4} \caption{A small example: for the ambient box $4 \times 3$, $g_{22} \cdot g_{22}=g_{3311}+g_{3221}+g_{2222}-g_{3321}-g_{3222}$.}\label{G22G22}
\end{figure}

 Consider the product $G_{\lambda} \cdot G_{\mu}$ in a $k \times (n-k)$ box for a rectangle of shape $\lambda=(\lambda_1^{\beta})$ and an arbitrary tableau of shape $\mu=(\mu_1^{\alpha_1},\ldots,\mu_m^{\alpha_m})$. Define the \textbf{upper rectangle} of $\mu$, \textbf{$\mu^{u}:=\mu_1^{\alpha_1}$}. Let $M^{u}$ denote the filling of the upper rectangle of $M$.

\begin{Lem}[Reduction Lemma]\label{redlem} Consider the poset of fillings of the two tableaux representing the product $G_{\lambda} \cdot G_{\mu}$ in the ambient box $k \times (n-k)$, where $\lambda=(\lambda_1^{\beta})$ and $\mu=(\mu_1^{\alpha_1},\ldots,\mu_m^{\alpha_m})$. Let $M^{u}=(1^{\tau_1},\ldots,(c+a)^{\tau_{\beta+\alpha_1}})$. Under $I_1$, the poset of fillings of $\mu$ reduces to the product of $\lambda'$ and $\mu'$, where $\lambda'=(\lambda \cup M^{u})|_{(\alpha_1+1,\ldots)}$ and $\mu'=\mu \setminus \mu^{u}=\mu|_{(2,\ldots, m)}$, with the product in ambient box
 $(k-\alpha_1) \times (\lambda_1+\tau_{\alpha_1})$.
\end{Lem}

\begin{proof} Apply $I_1$ to the upper rectangle of $\mu$. Then by Theorem \ref{2rect}, all terms are canceled except those whose upper rectangle filling matches that of $M$. Consider $w^+ := w(\lambda) \bigcup w(\mu^u)$ obtained from the standard filling of ($\lambda$) and this filling of the upper rectangle. Then the remaining poset is equivalent to the poset of the product $G_{\lambda'} \cdot G_{\mu'}$, where $\lambda'$ is the shape determined by a standard filling of the partition of shape corresponding to $w^+$ with the top $\alpha_1$ rows truncated, with the added restriction on the number of 1's to the number of $\alpha_1$'s in $w^+$.
 \end{proof}

Note that by our construction of the Reduction Lemma, $l(\lambda) \leq l(\lambda')$ and $l(\mu) < l(\mu')$. Thus we have reduced to a product of two smaller tableaux.

\begin{figure}[htbp]
\begin{center}
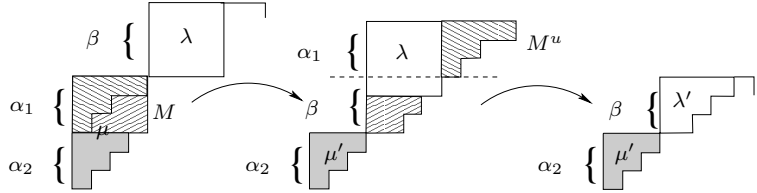
\caption{A graphical description of the Reduction Lemma.}\label{redlem}
\end{center}
\end{figure}

%---------------------------------------------------------------
\section{Applications to Stembridge Cases}
%---------------------------------------------------------------

Now, we will use our Reduction Lemma to prove Theorem \ref{sequencethm} in each of the Stembridge cases. As in the statement of the lemma, we use the notation $\lambda=(\lambda_1^{\beta})$ and $\mu=(\mu_1^{\alpha_1},\ldots,\mu_m^{\alpha_m})$, and the poset of fillings of the two tableaux representing the product $G_{\lambda} \cdot G_{\mu}$ in ambient box $k \times (n-k)$. We denote the filling of the minimal element $M$ (of shape $\mu$) as $(1^{\tau_1},\ldots,(\beta+\alpha_1)^{\tau_{\beta+\alpha_1}})$. We will successfully apply the Reduction Lemma to most of the Stembridge cases. We omit the discussion of the ambient box in the following proofs as it is not particularly enlightening, but the size of the ambient box after the lemma is applied is easily determined in each case.

%----------------------------------------------

\begin{Prop}\label{2a} [Pieri-Stembridge case (2a)] Theorem \ref{sequencethm} holds in the case of a single row and an arbitrary tableau $(\beta=1)$: let $\lambda=(\lambda_1)$ and $\mu=(\mu_1^{\alpha_1},\ldots,\mu_m^{\alpha_m})$. \end{Prop}

\begin{proof} We proceed by induction. Consider the base case $G_{(\lambda_1)} \times G_{(\mu_1)}$:
the only possible fillings for $(\mu_1)$ are  $[1,\ldots, 1, 2, \ldots, 2] \text{ or } [1,\ldots, 1, [12],2, \ldots, 2]$.
The involution matches terms with 2 or $[12]$ in the $i^{th}$ spot.\\

Now, assume the claim holds for the products $G_{(\lambda_1)} \times G_{\mu} \text{ for } l(\mu)=j \; \forall \; j < N$ for some $N \in \Z^+$.
Then consider $G_{(\lambda_1)}  \times  G_{\mu} \text{ for } l(\mu)=N$. Applying the Reduction Lemma,
we get the product $G_{\mu_2,\ldots, \mu_m} \times G_{(\lambda_1')} \text{ for } \lambda_1'=\tau_{\beta+\alpha_1}$. Note that $l(\lambda_1')=1$.
\end{proof}

%\begin{center} \begin{figure}
%\epsfig{file=1row.eps, scale=1} \caption{The Reduction Lemma applied to Pieri-Stembridge case (2a).}
%\end{figure}
%\end{center}
\begin{figure}[htbp]
\begin{center}
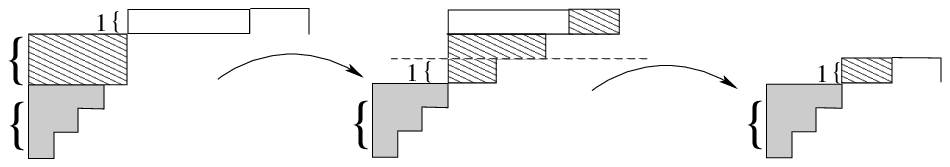
\caption{The Reduction Lemma applied to Pieri-Stembridge case (2a).}\label{2a}
\end{center}
\end{figure}

%----------------------------------------------
\begin{Prop}\label{2b} [Pieri-Stembridge case (2b)] Theorem \ref{sequencethm} holds in the case of a single column and an arbitrary tableau $(\lambda_1=1)$: let $\lambda=(1^{\beta})$ and $\mu=(\mu_1^{\alpha_1},\ldots,\mu_m^{\alpha_m}) $.\end{Prop}
\begin{proof}
We proceed by induction. Consider the base case $G_{(1^{\beta})}  \times G_{(\mu_1)}$:
the only possible fillings for $(\mu_1)$ are  $$[1,\ldots, 1] \text{ or }[1,\ldots, 1, 2] \text{ or } [1,\ldots, 1, [12]].$$ The involution matches the latter two cases with each other, leaving the first.\\
Assume the claim holds for the products $$G_{(1^{\beta})} \times G_{\mu} \text{ for } l(\mu)=j \; \forall \; j < N$$ for some $N \in \Z^+$.
Then consider $G_{\mu} \times G_{(1^{\beta})} \text{ for } l(\mu)=N$.
 Applying the Reduction Lemma, we get the product $G_{\mu_2,\ldots, \mu_m} \times G_{(1^{\beta'})}$ where $\beta'=\beta-\mu_1+\tau_{\beta+\alpha_1}$.
\end{proof}

%\begin{center} \begin{figure}
%\epsfig{file=1col.eps, scale=0.4} \caption{The Reduction Lemma applied to Pieri-Stembridge case (2b).}
%\end{figure}
%\end{center}

\begin{figure}[htbp]
\begin{center}
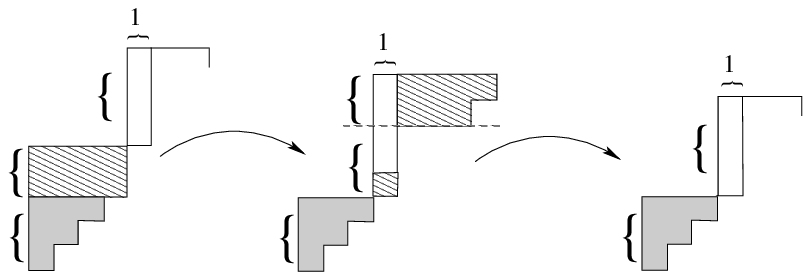
\caption{The Reduction Lemma applied to Pieri-Stembridge case (2b).}\label{2b}
\end{center}
\end{figure}

%----------------------------------------------
\begin{Prop}\label{3a} [Stembridge case (3a)] Theorem \ref{sequencethm} holds in the case of a rectangle and a left near rectangle $(\mu_2=1)$: let $\lambda=(\lambda_1^\beta)$ and  $\mu=(\mu_1^{\alpha_1}1^{\alpha_2})$. \end{Prop}
\begin{proof}
Apply the Reduction Lemma to the upper rectangle $\mu^+=\mu_1^{\alpha_1}$. Then $\mu'=1^{\alpha_2}$ and $\lambda'=(\lambda \bigcup \mu^{u})|_{\alpha_1+1,\ldots}$. This is Pieri-Stembridge case (2b).
\end{proof}

%\begin{center} \begin{figure}
%\epsfig{file=leftcol_fathook.eps, scale=0.4} \caption{The Reduction Lemma applied to Stembridge case (3a).}
%\end{figure}
%\end{center}

\begin{figure}[htbp]
\begin{center}
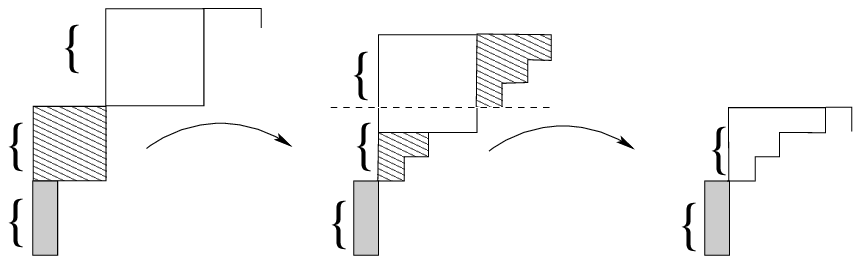
\caption{The Reduction Lemma applied to Stembridge case (3a).}\label{3a}
\end{center}
\end{figure}

%----------------------------------------------
\begin{Prop}\label{3b}[Stembridge case (3b)] Theorem \ref{sequencethm} holds in the case of a rectangle and a bottom near rectangle $(\alpha_2=1)$: let $\lambda=(\lambda_1^\beta)$ and $\mu=(\mu_1^{\alpha_1}\mu_2)$. \end{Prop}
\begin{proof}
Apply the Reduction Lemma to the upper rectangle $\mu^+=\mu_1^{\alpha_1}$. Then $\mu'=\mu_2$ and $\lambda'=(\lambda \bigcup \mu^{u})|_{\alpha_1+1,\ldots}$ . This is Pieri-Stembridge case (2a).
\end{proof}
%\begin{center} \begin{figure}
%\epsfig{file=bottomrow_fathook.eps, scale=0.4} \caption{The Reduction Lemma applied to Stembridge case (3b).}
%\end{figure}
%\end{center}

\begin{figure}[htbp]
\begin{center}
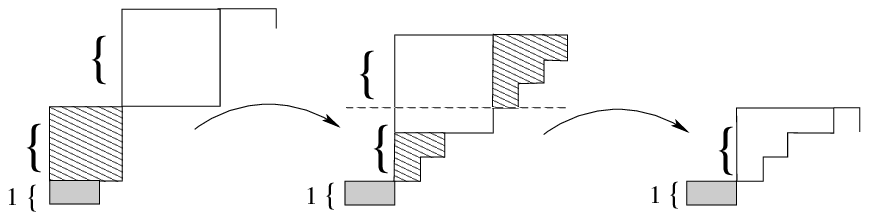
\caption{The Reduction Lemma applied to Stembridge case (3b).}\label{3b}
\end{center}
\end{figure}

%----------------------------------------------
\begin{Prop}\label{3c}[Stembridge case (3c)] Theorem \ref{sequencethm} holds in the case of a rectangle and a top near rectangle $(\alpha_1=1)$: let $\lambda=(\lambda_1^\beta)$ and $\mu=(\mu_1 \mu_2^{\alpha_2})$. \end{Prop}

\begin{proof}
Apply the Reduction Lemma to the upper rectangle $\mu^+=\mu_1$. Then $\mu'=\mu_2^{\alpha_2}$ and $\lambda'=\lambda_1^{\beta-1} \lambda_2$ for $\lambda_2 \leq min\{\mu_1,\lambda_1\}$. This is Pieri-Stembridge case (2b).
\end{proof}

%\begin{center} \begin{figure}
%\epsfig{file=toprow_fathook.eps, scale=0.5} \caption{The Reduction Lemma applied to Stembridge case (3c).}
%
%\end{figure}
%\end{center}

\begin{figure}[htbp]
\begin{center}
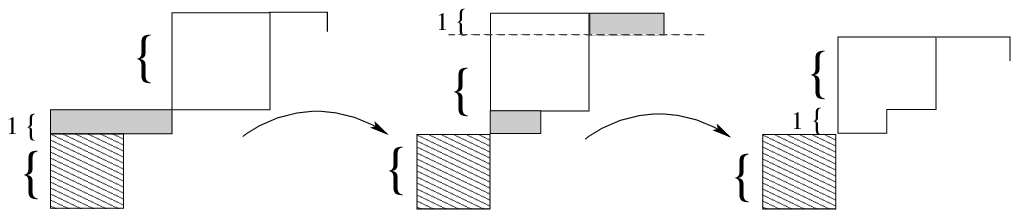
\caption{The Reduction Lemma applied to Stembridge case (3c).}\label{3c}
\end{center}
\end{figure}

%----------------------------------------------
\begin{Prop}\label{4a} [Stembridge case (4a)] Theorem \ref{sequencethm} holds in the case of a two row rectangle $(\beta=2)$ and a fat hook: let $\lambda=(\lambda_1^2)$ and $\mu=(\mu_1^{\alpha_1},\mu_2^{\alpha_2})$.
\end{Prop}
\begin{proof}
%Then the columns of the upper rectangle $\mu_1^{\alpha_1}$ must be of one of the following forms: $$[1,2,\ldots,\alpha_1]^{tr} \text{  or  } [1,3,4,\ldots,\alpha_1-1]^{tr} \text{  or  } [3,4,\ldots,\alpha_1-2]^{tr}$$.

Apply the Reduction Lemma to the upper rectangle $\mu^+=\mu_1^{\alpha_1}$. Then $\mu'=\mu_2^{\alpha_2}$ and $l(\lambda') \leq 2$. This is Stembridge case (3b) or (3c).
\end{proof}

%\begin{center} \begin{figure}
%\epsfig{file=2row.eps, scale=0.4} \caption{}
%\end{figure}
%\end{center}

\begin{figure}[htbp]
\begin{center}
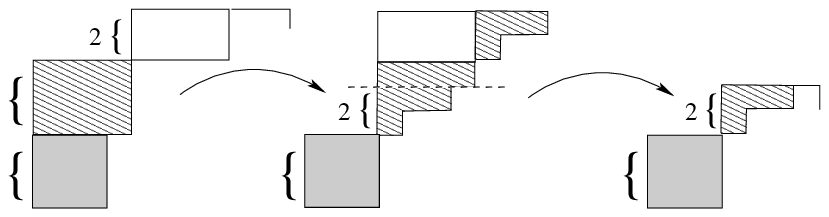
\caption{The Reduction Lemma applied to Stembridge case (4a).}\label{4a}
\end{center}
\end{figure}

%----------------------------------------------
\begin{Prop}\label{4b} [Stembridge case (4b)] Theorem \ref{sequencethm} holds in the case of a two column rectangle $(\lambda_1=2)$ and a fat hook : let $\lambda=(2^\beta)$ and $\mu=(\mu_1^{\alpha_1},\mu_2^{\alpha_2})$ .
\end{Prop}
\begin{proof}
%The columns of the upper rectangle $\mu_1^{\alpha_1}$ must be of the form $$[1,2,\ldots,\alpha_1]^{tr} \text{ or } [1,\ldots,m,\beta+1,\ldots,\beta+\alpha_1-m]^{tr}, m=1,\ldots,\beta-1$$
%The filling of the upper rectangle $\mu_1^{\alpha_1}$ is limited to 2 columns of the $2^{nd}$ type.

Apply the Reduction Lemma to the upper rectangle $\mu^+=\mu_1^{\alpha_1}$. Then $\mu'=\mu_2^{\alpha_2}$ and $\lambda'=2^{\beta_1} 1^{\beta_2}$ for some $0 \leq {\beta_1},{\beta_2} \leq \beta$. This is Stembridge case (3a).
\end{proof}

%\begin{center} \begin{figure}
%\epsfig{file=2col.eps, scale=0.4} \caption{The Reduction Lemma applied to Stembridge case (4b).}
%\end{figure}
%\end{center}

\begin{figure}[htbp]
\begin{center}
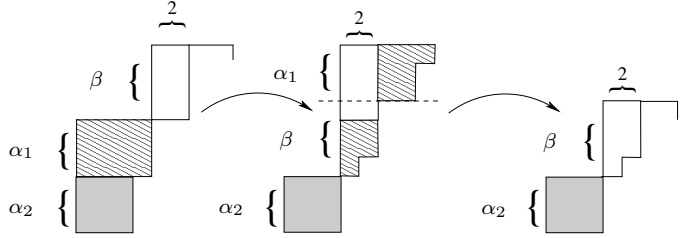
\caption{The Reduction Lemma applied to Stembridge case (4b).}\label{4b}
\end{center}
\end{figure}

%---------------------------------------------------------------
\section{The Last Case: Right Near Rectangle and Rectangle}\label{last}
%---------------------------------------------------------------

Finally, consider Stembridge case (3d), of a rectangle $\mu=(\mu_1^{\alpha_1})$, and a right near rectangle $\lambda =(\lambda_1^{\beta_1} (\lambda_1-1)^{\beta_2})$.

We consider this as ``almost" the rectangle times rectangle case, and we consider the fillings of the rectangle. As in the 2 rectangles case, we can characterize the fillings of the columns of the rectangle in terms of column blocks:
\begin{Lem} Consider the product of $\mu=(\mu_1^{\alpha_1})$ and $\lambda =(\lambda_1^{\beta_1} (\lambda_1-1)^{\beta_2})$ in an ambient box of size $k \times (n-k)$. A filling of a column in $\mu$ is built out of a a basic set of blocks, given by
 \begin{itemize}
 \item $[1,2,\ldots,p]^{tr}$ for $p \leq \alpha_1$.
 \item $ [\beta_1+\beta_2+1,\beta_1+\beta_2+2,...,q]^{tr}$ for $q \leq \beta_1+\beta_2+l$.
 \end{itemize}
plus one initial string of $[\beta_1+1,\beta_1+2,...,r]$ for $r \leq \beta_1+\beta_2$ which can occur of any length from 0 to $\beta_2$ running southwest through the filling, occurring in order in $w(T)$. Additionally, two overlaps are allowed in each column (either as two blocks with two entries each, or one block with 3 entries).\\
\end{Lem}
\begin{proof} A parallel argument to that in the 2-rectangles column characterization holds as to why the blocks occur and why there can be no other gaps besides a jump between blocks. The extra string must occur in order in the word, which means each entry must occur weakly southwest of the previous. (By semi-standardness, it must occur strictly south, weakly west.) Overlaps are allowed between the column blocks, as well as with entries of the snake.\\

\end{proof}

First we apply $I_1$, to the upper rectangles of the entries in the poset. All terms where $T^{u}$ differs from $M^{u}$ will cancel under this involution, leaving terms where $T^{u}$ is identical to $M^{u}$. Now, we define two more involutions, which when applied in succession will cancel the rest of the terms.

\begin{center}
\begin{figure}[htbp]
\epsfig{file=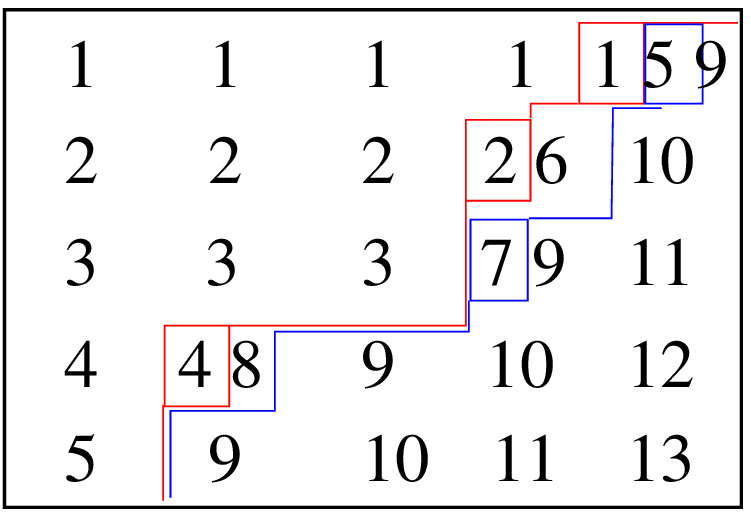, scale=0.45} \caption{The snake in a term in the product $G_{4^5}G_{5^4,4^4}$.}\label{fathooksnake}
\end{figure}
\end{center}

We call the \textbf{intersnake} those entries in the filling that correspond to the missing column piece in the fat hook; more explicitly, the entries that correspond to
\begin{itemize}
\item the first occurrence of entries $(\beta_1+1,\beta_1+2,\ldots,\beta_1+\beta_2)$
\item the $\lambda_1^{th}$ occurrence of entries $(\beta_1+\beta_2+1,\beta_1+\beta_2+2,\ldots)$
\end{itemize}
in the word of the tableau.

We extend our previous definition of a \textbf{snake} for this case to be the union of four lines: the line below all single entries $i$ occurring in box $(i,j)$, the line above all entries $\beta_1+\beta_2$, and the lines defining the intersnake region. Note that a snake uniquely defines a filling of a tableau, as the values in each region are determined.

\begin{Lem}\label{snakelem} The unique tableau in which the intersnake is maximally northeast is $M$, the first tableau in lexicographic order.  \end{Lem}
\begin{proof} We will show that any weakly southwest shift of a snake box in an arbitrary tableau $T$ will yield a tableau $T'$ that is lexicographically after $T$. Consider a column in $T$: $$[1,2,\ldots,p,s,\beta_1+\beta_2+1,\ldots,\beta_1+\beta_2+(\beta_1-p-1)]$$ for $s$ in the intersnake of $T$ and $p < \beta_1$, and a column to the left $$[1,2,\ldots,q,\beta_1+\beta_2+1, \beta_1+\beta_2+2,\ldots,\beta_1+\beta_2+(\beta_1-q)]$$ for $p< q \leq \beta_1$. If we move $s$ to the left column and shift appropriately to get another tableau $T'$ where the snake box is southwest of the box in $T$, we get columns in $T'$ of the form $$[1,2,\ldots,p,\beta_1+\beta_2+1,\ldots,\beta_1+\beta_2+(\beta_1-p-1),\beta_1+\beta_2+(\beta_1-p)]$$ and $$[1,2,\ldots,q,s,\beta_1+\beta_2+1, \beta_1+\beta_2+2,\ldots,\beta_1+\beta_2+(\beta_1-q-1)]$$ respectively. We get $w(T')$ by taking $w(T)$ and replacing $[\beta_1+\beta_2+(\beta_1-q)]$ with the larger $[\beta_1+\beta_2+(\beta_1-p)]$. That is, $T'$ comes after $T$ lexicographically.
\end{proof}

We will use \textbf{$\text{SS-SVT}^{\;I_k}$} to denote those semi-standard set valued tableaux left after involution $k$.
 \begin{Def} We define the function $I_2: \text{SS-SVT}^{\;I_1} \rightarrow \text{SS-SVT}^{\;I_1}$ as follows: Compare $M$ and $T$ along the intersnake as defined by $M$, and find the first box where they differ. Let this be the $k^{th}$ position in the intersnake of $M$. Let $L$ be the length of the intersnake.
\begin{itemize}

\item If $M(i,j)$ in $T(i,j)$, remove it as follows:
\begin{itemize}
\item if $k=L$, remove $M(i,j)$.
\item if $k<L$, remove $M(i,j)$ and replace $M(i,j)+i+1$ with $M(i,j)+i$ for $i=[0,L-k-1]$.
\end{itemize}

\item If $M(i,j)$ is not in $T(i,j)$, add it as follows:
\begin{itemize}
\item if $k=L$, add $M(i,j)$.
\item if $k<L$, add $M(i,j)$ to $T(i,j)$ and replace $M(i,j)+i$ with $M(i,j)+i+1$ for $i=[0, L-k-1]$.
\end{itemize}

\end{itemize}
\end{Def}

\begin{figure}[htbp]
\begin{center}
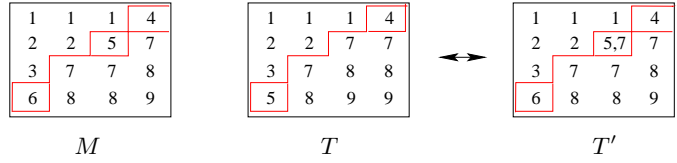
\caption{The minimal term and example of $I_2$, where $\lambda= (4,4,3,3), \mu = (4,4,4,4)$}\label{i2}
\end{center}
\end{figure}

\begin{Prop}\label{inv2} $I_2$ is a sign-reversing involution.
\end{Prop}
\begin{proof} This function changes the content of a tableau $T$ by either adding or removing the last value in the intersnake of $T$. This is clearly sign-reversing, as we are changing the size of the content by one. Note that since we are applying this to terms in which the above-snake region matches $M$, now the only overlaps in the columns are of the form $[s,\beta_1+\beta_2+1]$ for $\beta_1+1 \leq s \leq \beta_1+\beta_2$. Then, $T(i,j)=[s,\beta_1+\beta_2+1]$ or $T(i,j)=[\beta_1+\beta_2+1]$. The shifting is necessary to get a tableau that satisfies the reverse lattice word condition.
\end{proof}

Note that we apply $I_2$, we can always remove an entry from the intersnake, but we cannot always add one: tableaux which are not matched by this involution are those in which
\begin{itemize}
\item $I_2$ wants to add an entry
\item we cannot pull the intersnake through, because the full intersnake already appears in $T$.
\end{itemize}

Then the only terms left in the poset are those with $T^{u}=M^{u}$, and the full intersnake appears in $T$ but does not match $M$'s intersnake. We will define one last involution on the remaining terms.

\begin{Def} We define the function $I_3: \text{SS-SVT}^{\;I_2} \rightarrow \text{SS-SVT}^{\;I_2}$ as follows: Compare $M$ and $T$ along the intersnake as defined by $M$, and find the first box where they differ. Let this be the $k^{th}$ position in the snake of $M$, and call the box in $T$, $(i,j)$, and in $M$,$(i_M,j_M)$ . In column $j$,
\begin{itemize}
\item If every entry in the column is single, let $(i',j)$=box above entry $\beta_1+\beta_2+1$ ($\i'=\beta$ if $\beta_1+\beta_2+1$ is not there). Then $T'(i',j)=[T(i',j),\beta_1+\beta_2+1]$ and replace entry $(k,j)$ with $(k,j)+1$ for $k=[i'+1,l-i']$.
\item If box $(i',j)$ has two entries, $T'(i',j)=T(i',j)\backslash \beta_1+\beta_2+1$ and replace $(k,j)$ with $(k,j)-1$ for $k=[i'+1,l-i']$.
\end{itemize}
Then $I_3$ matches $T$ with $T'$.
\end{Def}
This function changes the content of a tableau $T$ by shifting one column in the below-snake region.

\begin{figure}[htbp]
\begin{center}
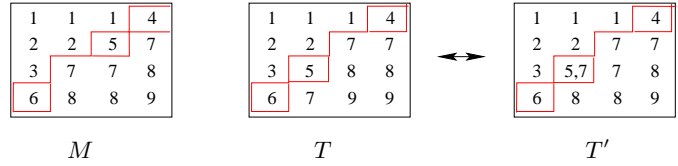
\caption{The minimal term and example of $I_3$, where $\lambda= (4,4,3,3), \mu = (4,4,4,4)$.}\label{i3}
\end{center}
\end{figure}

\begin{Prop}\label{inv3} $I_3$ is a sign-reversing involution.
\end{Prop}
\begin{proof}  The choices of $i,j,$ and $j'$ are well-defined. It is clearly sign-reversing, as we are changing the size of the content by one.\\
Note that while we can always shift down (remove the last entry in) a column to get another valid filling, we cannot always shift up (add the last entry). Let us consider those cases in which we shift up. If $j<\beta$, there is a $\beta_1+\beta_2+1$ in column $i$, so we can clearly add the last term to the shifted column. However, if $j=\beta$, adding $\beta_1+\beta_2+1$ might violate the reverse lattice word condition. Since each snake box occurs strictly south of the previous, $j=\beta$ implies that $k$ is the last snake entry. Note that by our choice of box $(i',j)$, $T$ matches $M$ in columns east of column $j_M$ and in any columns in between $j$ and $j_M$, so we need only consider those two columns. By lemma \ref{snakelem}, $(i',j)$ is strictly southwest of $(i_M,j_M)$. This implies that $i_M > \beta$, so both columns $j_M$ and $j$ have $\beta_1+\beta_2+1$ in $M$, which means that the reverse lattice word condition is not violated by two occurrences and we can thus add it to $T(i',j)$.
\end{proof}

\begin{Prop} [Stembridge case (3d)] Theorem \ref{sequencethm} holds in the case of a rectangle and a right near rectangle. \end{Prop}

\begin{proof}
As shown, by applying $I_1$, $I_2$, and $I_3$ in that order, all terms are canceled except for $M$.
\end{proof}

Let us now go back to Theorem \ref{sequencethm}.
\begin{proof}[Proof of Theorem \ref{sequencethm}.]
Recall that we want to construct the exact sequence
\begin{equation}\label{eq:seq}
0 \rightarrow \C^1 \rightarrow \cdots \mapright{f_{k-1}}  \bigoplus_{T, |T|=|\nu|+ k-1, content(T) \subseteq \lambda} \C^1 \mapright{f_k} \cdots
\end{equation}
We claim that the functions are given by
$$f_i(\nu)=\left\{ \begin{array}{cc}
 \nu' &\text{ if } |\nu'| = |\nu|+1, I(\nu)=\nu' \text{ for some involution} \\
  0 &\text{ otherwise }
\end{array} \right. $$

From our involutions, we have the following two-term exact sequences, where each tableau filling $T$ is represented by $\C$. For each pair of tableaux of sizes $|\nu|+k-1$ and $|\nu|+k$ and matched by one of the above involutions, we define the functions in our sequence by:
$$f_i=\left\{ \begin{array}{cc}
  Id &\text{ if } i=k \\
  0 &\text{ otherwise }
\end{array} \right.$$
for $Id$ the identity function. That is, we obtain a sequence
$$ 0 \rightarrow  \cdots \rightarrow 0  \rightarrow \C^1 \mapright{f_k} \C^1 \rightarrow 0 \cdots \rightarrow 0$$
for each pair of tableaux. Each sequence is clearly exact. For a given $\lambda$, we can construct the desired sequence as a direct sum of sequences of this form, one corresponding to each matched pair in the order ideal of $\lambda$.

\end{proof}

%---------------------------------------------------------------
\section{An Extension to the Thomas-Yong Cases}
%---------------------------------------------------------------
Consider partitions $\lambda$ and $\mu$ in a $(k \times (n-k))$ box. We will review the notation introduced in \cite{TY}. We call $R=(\lambda,\mu,k \times (n-k))$ a \textbf{Richardson quadruple}, and use the notation $poset(R)$ to denote the associated poset of fillings of $\mu$. Place $\lambda$ in the upper left corner of the box, then rotate $\mu$ by $180 ^\circ$ (call this \textbf{rotate($\mu$})) and place it in the lower right corner. This quadruple $(\lambda,\mu,k\times (n-k))$ is called \textbf{basic} if $\lambda \bigcup \text{rotate}(\mu)$ does not contain any full rows or columns. If it is not basic, we can remove all full rows and columns to get a \textbf{basic demolition} $({\widetilde \lambda},{\widetilde \mu},{\widetilde k}\times {(\widetilde{n}-\widetilde{k})})$. We call each row (column) removal a \textbf{row (column) demolition}. Notice that if $\lambda \bigcap \text{rotate}(\mu) \neq \emptyset$, then $G_\lambda \cdot G_\mu = 0$.
\begin{Thm}\cite{TY} A Richardson quadruple is multiplicity-free if and only if its basic Richardson quadruple is multiplicity-free.
If a basic Richardson quadruple $(\lambda,\mu,k\times n-k)$ is multiplicity-free, then it must be in the cases classified by \cite{Stem}. \end{Thm}

For example, consider the case $((4,4,2,2,1),(4,3,2,1),5 \times 5)$. This product is not multiplicity-free, but has a basic demolition of $(1,1,2 \times 2)$, which is multiplicity-free.\\
\begin{center} \begin{figure}[htbp]
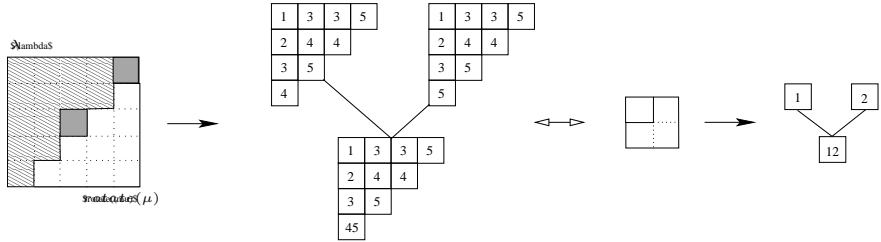 \caption{A comparison of $((4,3,2,1),(4,4,2,2,1),5 \times 5)$ and $(1,1,2 \times 2)$, where columns 1, 2, and 3 and rows 2, 4, and 5 have been removed.}
\end{figure}
\end{center}
We will show that our analysis of the \cite{Stem} multiplicity-free cases extends to this larger class of products by showing that the posets of a Richardson quadruple and its basic demolition are isomorphic. Let us define the \textbf{accessible word $w_A$} as the independent values of $\lambda$ read in increasing order, or equivalently $$w_A(j)=1+\#(\text{rows of } \lambda \text{ in column } n-k-j).$$

\begin{Lem}[Column Demolition Lemma]\label{coldem} For $R=(\lambda,\mu,k \times (n-k))$, if column $c$ is full, then $poset( \widetilde{R})$  is isomorphic to $poset(R)$. \end{Lem}

\begin{proof} We will show that the function given by removal of column $c$ from every tableau in the poset gives the isomorphism.\\
Let $l$ be the number of rows in column $c$ in $\lambda$, so $k-l$ is the number of rows in column $c$ in $\mu$. We claim that a full column implies that in every filling of $\mu$, column $c$ is filled as $[l+1,\ldots,k]$. Consider box $(1,c)$ in $\lambda$: filling it with a value greater than $l+1$ would mean that the last box in the column $(c,k-l)$ must contain an entry that is at least $k+1$, which is larger than the ambient box. Filling it with a value less than $l+1$ would lead to an invalid filling of boxes $(1,1)$ to $(1,c-1)$. Then $(1,c)$ contains the single entry $l+1$, and in order to fit in the ambient box, the proposed column filling is the only possibility.

We define a function $F:poset(R) \rightarrow poset(\widetilde{R})$ as follows: for each tableau $T$ in $poset(R)$, we remove column $c$ and shift columns $c+1$ to $\mu_1$ left by one, where entries are shifted with their associated boxes. This operation does not change semistandardness. Column $c$ is independent in the sense of the reverse-lattice word, and since the ambient box has $k$ rows, there cannot be a $k+1$ that depends on it. Thus for $T$ of shape $(\mu_1,\ldots, \mu_m)$ we get a valid tableau filling $T'$ of shape $(\mu_1-1,\ldots,\mu_{k-l}-1,\mu_{k-l+1},\ldots,\mu_{m})$ in $poset(\widetilde{R})$ . The inverse function $F^{-1}:\widetilde{R} \rightarrow R$ shifts columns $c$ to $mu_1-1$ to the right by one and adds the column back in. This is clearly an isomorphism.

\end{proof}

%If there is a full column from $\lambda$, removing it will clearly give an identical poset. If there is a full column from $\mu$, it must have the superstandard filling in every tableau.\\

Let $\lambda=(\lambda_1,\ldots,\lambda_l)$ and $\mu=(\mu_1,\ldots,\mu_m)$. There is a full row in the diagram if and only if $\lambda_r + \mu_{k-r+1} = n-k$.

\begin{Lem}\label{fulltoprow} For $R=(\lambda,\mu,k \times (n-k))$, if $\lambda_1=n-k$, then $poset(R)$ is isomorphic to $poset(\widetilde{R})$ for $\widetilde{R}=(\lambda|_{(2,\ldots,l)},\mu,(k-1) \times (n-k))$.
\end{Lem}
\begin{proof} First we observe that $\lambda_1=n-k$ implies that $\mu$ cannot have any 1's in any filling. Then we define the function $F: R \rightarrow \widetilde{R}$ as the function that takes each entry $i$ to entry $i-1$, with obvious inverse $F^{-1}$. $F$ is clearly an isomorphism.
\end{proof}

\begin{Lem}[Row Demolition Lemma]\label{rowdem} For $R=(\lambda,\mu,k \times (n-k))$, if row $r$ is full, then the $poset(\widetilde{R})$ is isomorphic to $poset(R)$. \end{Lem}

\begin{proof} We will show that the function given by removal of the first $k-r+1$ boxes in row 1 of $\mu$, followed by a shifting up of the first $k-r+1$ columns, gives the isomorphism.

\begin{figure}[htbp]
\begin{center}
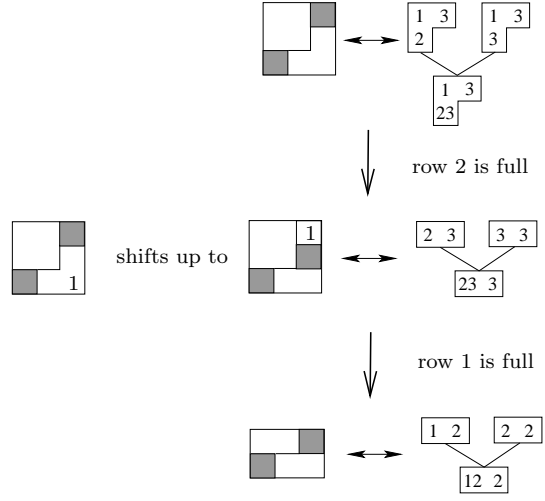
\caption{The demolition of row 2 in $((2,2),(2,1),3 \times 3)$ to $((2),(2),2 \times 3)$.}\label{rowdemo}
\end{center}
\end{figure}

Let $w$ be the number of columns in row $r$ in $\lambda$, so $n-k-w$ is the number of columns in row $r$ in $\mu$. First we claim that the entry in $(1,\mu_{k-r+1})$ is weakly less than $r$. There are only $k-r+1$ accessible entries that satisfy this condition, so these must be the filling of this block in every $\mu$, for $j$ from 1 to $\mu_m$ (i.e. box $(1,j)$ contains the single entry $w_A(j)$). Then we want to remove this extraneous information from the fillings of $\mu$ as follows: take out boxes $(1,j)$ for $j$ from 1 to $k-r+1$, then shift up columns 1 to $k-r+1$ by one box. This will give a new tableau of shape $\mu'=\mu \backslash \mu_{k-r+1}$. We instead record this information in $\lambda'=\lambda \bigcup w_A(1,\cdots,\mu_m)$. That is, $(\lambda,\mu,k \times (n-k))$ is now the product $(\lambda',\mu',k \times (n-k))$.

First we show that semistandardness is preserved. The only place problems may occur are along the vertical line between columns $\mu_{k-r+1}$ and $\mu_{k-r+1}+1$. We start at the top of the column. Consider box $(1, \mu_{k-r+1})$ with entry $a \leq r$.  Since $w_A(k-r+2) \geq r+1$, box $(1, \mu_{k-r+1}+1)$ must be filled with entries $\geq r+1$. Next we consider box $(2, \mu_{k-r+1})$: we claim that the entries are weakly less than $r+1$. If it contained entry $r+2$, then by semistandardness, the last box in the column $(k-r+1, \mu_{k-r+1})$ would have to have entries that were at least $k+1$, which would exceed the ambient box. When we shift box $(2, \mu_{k-r+1})$ to $(1, \mu_{k-r+1})$, semistandardness in row 1 is preserved. Next, assume for contradiction that the shifting violates semistandardess in a block
$\left[\begin{array}{cc}        a&b\\        c&d\end{array}\right]$;
that is, that $c>b$. But, semistandardness stipulates that $d \geq c$ and $b>d$, which together give the contradiction. Then semistandardness will be preserved throughout the column.

Next, we show that the reverse lattice condition still holds after the shift. Assume that entry $(i,j)$ depends on entry $(i',j')$. The only place where the shifting will affect the order is if $i'=i+1$, $j<k-r+1$ and $j'>k-r+1$. However, if this condition existed, semistandardness would be violated by shifting, which we have already shown cannot happen. Thus, this shifting gives a valid tableau filling.

Finally, we consider how this shifting affects the partial order. Since we are removing an identical block from every tableau, the relative content of any pair of tableaux is the same, and thus the relations in the poset are unchanged. By construction, $\lambda'_1=k$, so we can apply Lemma \ref{fulltoprow} to reduce the ambient box. Each step of the function is easily reversed, so it is an isomorphism.
\end{proof}

We note that the Row and Column Demolition Lemmata define commutative operations on the poset of a Richardson quadruple. Figure \ref{tydemo} is an example of a case with both a full row and column. (This product is Stembridge multiplicity-free in any ambient box, but is a good example of the row and column demolition commutativity.)

\begin{figure}[htbp]
\begin{center}
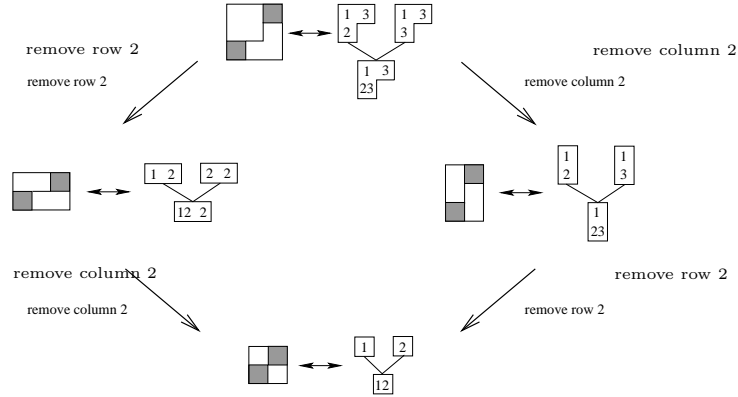
\caption{Two demolition paths of $((2,2),(2,1),3 \times 3)$ to $((1),(1),2 \times 2)$.}\label{tydemo}
\end{center}
\end{figure}

\begin{Prop} Theorem \ref{sequencethm} holds for any Richardson quadruple whose basic demolition is a Stembridge case. \end{Prop}
\begin{proof}
This follows clearly from the previous two lemmata.
\end{proof}

%---------------------------------------------------------------
\section{Acknowledgements}
%---------------------------------------------------------------
I would like to extend special thanks to Allen Knutson for both the statement of the question and continuing guidance throughout the process.

\bibliographystyle{alpha}

\end{document}